\documentclass[reqno]{amsart}
\usepackage{amsmath, amsthm, amssymb, amstext}

\usepackage[left=3.5cm,right=3.5cm,top=3.5cm,bottom=3.5cm,includeheadfoot]{geometry}
\usepackage{hyperref,xcolor}
\hypersetup{
 pdfborder={0 0 0},
 colorlinks,
}
\usepackage{accents}
\usepackage{enumitem}
\newtheorem{theorem}{Theorem}
\newtheorem{remark}[theorem]{Remark}
\newtheorem{lemma}[theorem]{Lemma}
\newtheorem{proposition}[theorem]{Proposition}

\DeclareMathOperator*{\ww}{w}         %
\DeclareMathOperator*{\Ss}{S}         %

\newcommand{\N}{\mathbb{N}}

\newcommand{\R}{\mathbb{R}}

\newcommand{\RN}{\mathbb{R}^N}

\newcommand{\Lp}[1]{L^{#1}(\Omega)}

\newcommand{\Wp}[1]{W^{1,#1}(\Omega)}

\newcommand{\lan}{\langle}
\newcommand{\ran}{\rangle}
\newcommand{\eps}{\varepsilon}
\newcommand{\ph}{\varphi}
\newcommand{\Om}{\Omega}

\newcommand{\into}{\int_{\Omega}}
\newcommand{\weak}{\overset{\ww}{\to}}

\newcommand{\close}{\overline{\Omega}}

\renewcommand{\l}{\left}
\renewcommand{\r}{\right}

\numberwithin{theorem}{section} \numberwithin{equation}{section}

\title[Robin fractional problems with symmetric variable growth]{Robin fractional problems with symmetric variable growth}

\author[A. Bahrouni]{Anouar Bahrouni}
\address[A. Bahrouni]{Mathematics Department, University of Monastir, Faculty of Sciences, 5019 Monastir, Tunisia}
\email{bahrounianouar@yahoo.fr}

\author[V.D. R\u{a}dulescu]{Vicen\c{t}iu D.\,R\u{a}dulescu}
\address[V.D. R\u{a}dulescu]{Faculty of Applied Mathematics, AGH University of Science and Technology, 30-059 Krak\'ow, Poland \& Department of Mathematics, University of Craiova, 200585 Craiova, Romania}
\email{radulescu@inf.ucv.ro}

\author[P. Winkert]{Patrick Winkert}
\address[P. Winkert]{Technische Universit\"{a}t Berlin, Institut f\"{u}r Mathematik, Stra\ss e des 17 Juni 136, 10623 Berlin, Germany}
\email{winkert@math.tu-berlin.de}


\begin{document}

\begin{abstract}
    In this paper we study the fractional $p(\cdot,\cdot)$-Laplacian and we introduce the corresponding nonlocal conormal derivative for this operator. We prove basic properties of the corresponding function space and we establish a nonlocal version of the divergence theorem for such operators. In the second part of this paper, we prove the existence of weak solutions of corresponding $p(\cdot,\cdot)$-Robin boundary problems with sign-changing potentials by applying variational tools.
\end{abstract}

\maketitle

\section{Introduction}

In recent years equations with nonstandard growth and related nonlocal equations have been studied by several authors. Such equations are very powerful and have lots of applications to different nonlinear problems including phase transitions, thin obstacle problem, stratified materials, anomalous diffusion, crystal dislocation, soft thin films, semipermeable membranes and flame propagation, ultra-relativistic limits of quantum mechanics, multiple scattering, minimal surfaces, material science, water waves and so. For a comprehensive introduction to the study of nonlocal problems and the use of variational methods in the treatment of these problems, we refer to the monograph by Molica Bisci, R\u{a}dulescu and Servadei \cite{Molica-Bisci-Radulescu-Servadei-2016}. The starting point in the study of nonlocal problems is due to the pioneering papers of Caffarelli, Roquejoffre and Sire \cite{Caffarelli-Roquejoffre-Sire-2010}, Caffarelli, Salsa and Silvestre \cite{Caffarelli-Salsa-Silvestre-2008}, and Caffarelli and Silvestre \cite{Caffarelli-Silvestre-2007} about the fractional diffusion operator $(-\Delta)^s$ for $s\in (0,1)$. Based on this, several other works have been published in the nonlocal framework. We refer, for example, to the works of Autuori and Pucci \cite{Autuori-Pucci-2013}, Bahrouni \cite{Bahrouni-2017}, Molica Bisci and R\u{a}dulescu \cite{Molica-Bisci-Radulescu-2015}, Pucci, Xiang and Zhang \cite{Pucci-Xiang-Zhang-2015}, R\u adulescu, Xiang and Zhang \cite{esaim, cvpde,ana}, and the references therein.

In this paper, we study the fractional $p(\cdot,\cdot)$-Laplace operator and we introduce the corresponding nonlocal conormal derivative for this operator. Kaufmann, Rossi and Vidal \cite{Kaufmann-Rossi-Vidal-2017} were the first who established some results on fractional Sobolev spaces with variable
exponent of the form $W^{ s,q(\cdot),p(\cdot,\cdot)}(\Omega)$ as well as properties of the fractional $p(\cdot,\cdot)$-Laplacian. In particular, it is shown that theses spaces are compactly embedded into variable exponent Lebesgue spaces. They also give an existence result for nonlocal problems involving the fractional $p(\cdot, \cdot)$-Laplacian. Bahrouni and R\u{a}dulescu \cite{Bahrouni-Radulescu-2018} obtained some further qualitative
properties of the fractional Sobolev spaces and the fractional $p(\cdot,\cdot)$-Laplacian. Further developments have been done by Bahrouni \cite{Bahrouni-2018} and Bahrouni and Ho \cite{Bahrouni-Ho-2020}, see also Ho and Kim \cite{Ho-Kim-2019}.

In this work, we continue the study of this new class of problems.
Our main aim is to investigate for the first time fractional $p(\cdot,\cdot)$-Laplacian equation with nonlocal Robin boundary condition. Precisely, we consider the problem
\begin{align}\label{eq}
    \begin{aligned}
    \l(-\Delta\r)^s_{p(\cdot,\cdot)}u+|u|^{\overline{p}(x)-2}u&=f(x,u)\qquad && \text{in } \Omega, \\
    \mathcal{N}_{s,p(\cdot,\cdot)}u+\beta(x)|u|^{\overline{p}(x)-2}u&= g(x) && \text{in } \R^N \setminus \close,
    \end{aligned}
\end{align}
where $\Omega\subset \R^{N}$, $N > 1$, is a bounded domain with Lipschitz boundary, $s\in (0,1)$, $p\colon \R^{2N} \to (1,+\infty)$ is a symmetric, continuous function bounded away from $1$, $\overline{p}(\cdot)=p(\cdot,\cdot)$, $g\in L^{1}(\R^{N}\setminus \Omega)$, $\beta \in L^{\infty}(\R^{N}\setminus \Omega)$ with $\beta \geq 0$ in $\R^{N}\setminus \Omega$ and $\l(-\Delta\r)^s_{p(\cdot,\cdot)}$ stands for the fractional $p(\cdot,\cdot)$-Laplacian which is given by
\begin{align}\label{operator}
    \l(-\Delta\r)^s_{p(\cdot,\cdot)}u(x)=\text{p.\,v.}\int_{\R^N}\frac{|u(x)-u(y)|^{p(x,y)-2}(u(x)-u(y))}{|x-y|^{N+sp(x,y)}}\,dy \quad \text{for } x\in\Omega.
\end{align}
Furthermore, $\mathcal{N}_{s,p(\cdot,\cdot)}$ is defined by
\begin{align}\label{Neumann boundary}
    \mathcal{N}_{s,p(\cdot,\cdot)}u(x)=\int_{\Omega}\frac{|u(x)-u(y)|^{p(x,y)-2}(u(x)-u(y))}{|x-y|^{N+sp(x,y)}}\,dy \quad \text{for }x\in\R^N\setminus\overline{\Omega},
\end{align}
and denotes the nonlocal normal $p(\cdot,\cdot)$-derivative (or $p(\cdot,\cdot)$-Neumann boundary condition) and describes the natural Neumann boundary condition in presence of the fractional $p(\cdot,\cdot)$-Laplacian. This work extends the notion of the nonlocal normal derivative introduced by Dipierro, Ros Oton and Valdinoci \cite{Dipierro-Ros-Oton-Valdinoci-2017} for the fractional Laplacian (see also Guan \cite{Guan-2006}), and  Mugnai and  Lippi \cite{Mugnai-Proietti-Lippi-2019} for the fractional $p$-Laplacian (see also Warma \cite{Warma-2016}). In the context of the fractional $p(\cdot,\cdot)$-Laplacian we also refer to the recent works of Mezzomo, Bonaldo, Miyagaki and Hurtado \cite{Mezzomo-Bonaldo-Miyagaki-Hurtado-2020} and Zhang and Zhang \cite{Zhang-Zhang-2020}.

This paper is organized as follows. In Section \ref{section_2} we recall some definitions and fundamental properties of the spaces $L^{p(\cdot)}(\Omega)$, $W^{1,p(\cdot)}(\Omega)$ and $W^{s, q(\cdot), p(\cdot,\cdot)}(\Omega)$. In Section \ref{section_3} we introduce the corresponding function space for weak solutions of \eqref{eq}, prove some properties and state the corresponding Green formula for problems like \eqref{eq}. In the last part, Section \ref{section_4}, we prove an existence result for problem \eqref{eq} with sign-changing potential based on the new results obtained in Section \ref{section_3} and by applying variational tools.

\section{Preliminaries}\label{section_2}

In this section, we recall some necessary properties of variable
exponent spaces and fractional Sobolev spaces with variable
exponent.

Suppose that $\Om$ is a bounded domain in $\RN$ with Lipschitz
boundary $\partial \Omega$ and let $p \in C_+(\overline{\Om})$, where
\begin{align*}
    C_+(\overline\Omega)=\l\{p\in C(\overline\Omega)\, : \, p(x)>1 \text{ for all }x\in\overline\Omega\r\}
\end{align*}
We set $p^-:=\min_{x \in \overline{\Omega}} p(x)$ and $p^+:=\max_{x \in \overline{\Omega}} p(x)$, then $p^->1$ and $p^+<\infty$. By $\Lp{p(\cdot)}$ we identify the variable
exponent Lebesgue space which is defined by
\begin{align*}
    \Lp{p(\cdot)} = \left \{u \ \Big | \ u: \Omega \to \R \text{ is measurable and } \into |u|^{p(x)}\,dx< \infty \right \}
\end{align*}
equipped with the Luxemburg norm
\begin{align*}
    \|u\|_{p(\cdot)} = \inf \left \{ \tau >0 : \into \left |\frac{u(x)}{\tau} \right |^{p(x)}\,dx \leq 1  \right \}.
\end{align*}
The variable exponent Sobolev space $\Wp{p(\cdot)}$ is defined by
\begin{align*}
    \Wp{p(\cdot)}=\l \{u \in \Lp{p(\cdot)}  :  |\nabla u| \in \Lp{p(\cdot)} \r\}
\end{align*}
with the norm
\begin{align*}
    \|u\|_{1,p(\cdot)}= \|\nabla u \|_{p(\cdot)}+\|u\|_{p(\cdot)}.
\end{align*}
For more information and basic properties of variable exponent
spaces we refer the reader to the papers of Fan and Zhao
\cite{Fan-Zhao-2001}, Kov{\'a}{\v{c}}ik and R{\'a}kosn{\'{\i}}k
\cite{Kovacik-Rakosnik-19991} and the monographs of Diening, Harjulehto, H\"{a}st\"{o} and R$\mathring{\text{u}}$\v{z}i\v{c}ka \cite{Diening-Harjulehto-Hasto-Ruzicka-2011} and R\u{a}dulescu and Repov\v{s} \cite{Radulescu-Repovs-2015}.

Let $L^{q(\cdot)}(\Omega)$ be the conjugate space of $L^{p(\cdot)}(\Omega)$, that is, $1/p(x)+1/q(x)=1$ for all $x \in \close$. If $u\in  L^{p(\cdot)}(\Omega)$ and $v\in L^{q(\cdot)}(\Omega)$, then  the H\"older-type inequality
\begin{align*}
    \left|\int_\Omega uv\,dx\right|\leq\left(\frac{1}{p^-}+ \frac{1}{q^-}\right)\|u\|_{p(\cdot)}\|v\|_{q(\cdot)}
\end{align*}
is satisfied. More general, if $p_j\in C_+(\overline\Omega)$ ($j=1,2,\ldots, k$) and
\begin{align*}
    \frac{1}{p_1(x)}+\frac{1}{p_2(x)}+\cdots +\frac{1}{p_k(x)}=1,
\end{align*}
then, for all $u_j\in L^{p_j(\cdot)}(\Omega)$ ($j=1,\ldots ,k$), we have
\begin{align*}
    \left|\int_\Omega u_1u_2\cdots u_k\,dx\right|\leq\left(\frac{1}{p_1^-}+ \frac{1}{p_2^-}+\cdots +\frac{1}{p_k^-}\right)\|u_1\|_{p_1(\cdot)}\|u_2\|_{p_2(\cdot)}\cdots \|u_k\|_{p_k(\cdot)}\,.
\end{align*}

In order to work variable Lebesgue and Sobolev spaces we need to consider the corresponding modular function. To this end, let $\rho\colon \Lp{p(\cdot)} \to \R $ be defined by
\begin{align*}
    \rho(u)=\int_{\Omega}|u|^{p(x)}\,dx.
\end{align*}

\begin{proposition}\label{pprp1}
    The following hold:
    \begin{enumerate}
    \item[(i)]
        $\|u\|_{p(\cdot)}<1\ (=1,\,>1)\ \Longleftrightarrow \ \rho(u)<1\ (=1,\>1)$;
    \item[(ii)]
        $\|u\|_{p(\cdot)}>1\ \Rightarrow\ \|u\|_{p(\cdot)}^{p^{-}}\leq \rho(u) \leq  \|u\|_{p(\cdot)}^{p^{+}}$;
    \item[(iii)]
        $\|u\|_{p(\cdot )}<1 \ \Rightarrow \ \|u\|_{p(\cdot)}^{p^{+}}\leq \rho(u) \leq \|u\|_{p(\cdot)}^{p^{-}}$\ .
    \end{enumerate}
\end{proposition}

\begin{proposition}
    If $u,u_{n}\in L^{p(\cdot)}(\Omega)$ with $n\in \mathbb{N}$, then the following statements are equivalent:
    \begin{enumerate}
    \item[(i)]
        $ \lim_{n\to +\infty} \|u_{n}-u\|_{p(\cdot)}=0$;
    \item[(ii)]
        $ \lim_{n\to +\infty} \rho(u_{n}-u)=0$;
    \item[(iii)]
        $u_{n}(x)\to u(x)$ a.\,e.\,in $\Omega$ and
        $ \lim_{n\to +\infty}\rho(u_{n})=\rho(u).$
    \end{enumerate}
\end{proposition}

Let us now introduce the fractional Sobolev space with variable exponents following the work of Kaufmann, Rossi and Vidal \cite{Kaufmann-Rossi-Vidal-2017}. To this end, let $s\in (0,1)$ and let $q\colon \close\to (1,\infty)$ and $p\colon \close\times\close\to (1,\infty)$ be two continuous functions.  Furthermore, we suppose that $p$ is symmetric and that both functions, $q$ and $p$, are bounded away from $1$, that is,
\begin{align}\label{cond_pq}
    \begin{split}
    p(x,y)&=p(y,x)\quad\text{for all }x,y\in \close\\
    1<q_-&:= \min_{x \in \overline{\Omega}} q(x)\leq q(x) \leq   q_+:= \max_{x \in \overline{\Omega}} q(x)\\
    1<p^-&:= \min_{(x,y) \in \overline{\Omega}\times \overline{\Omega}}p(x,y)\leq p(x,y)
    \leq p^+:= \max_{(x,y) \in \overline{\Omega}\times \overline{\Omega}} p(x,y).
    \end{split}
\end{align}

Now we introduce the fractional variable Sobolev space $W:=W^{s,q(\cdot),p(\cdot,\cdot)}(\Omega)$  which is given by
\begin{align*}
    W=\l\{u\colon\Omega\to\R \ \big | u\in L^{q(\cdot)}(\Omega),
    \into\into \frac{|u(x)-u(y)|^{p(x,y)}}{\lambda^{p(x,y)} |x-y|^{N+sp(x,y)}}\,dx\,dy <\infty\text{ for some }   \lambda>0 \r\}
\end{align*}
equipped with the variable exponent seminorm
\begin{align*}
    [u]_{s,p(\cdot,\cdot),\Omega}=\inf\l\{\lambda>0\ : \ \int_{\Omega \times \Omega}
    \frac{|u(x)-u(y)|^{p(x,y)}}{\lambda^{p(x,y)} |x-y|^{N+sp(x,y)}}\,dx\,dy <1\r\}.
\end{align*}
If we endow $W$ with the norm
\begin{align*}
    \|u\|_{W}=[u]_{s,p(\cdot,\cdot),\Omega}+\|u\|_{L^{q(\cdot)}(\Omega)},
\end{align*}
then $W$ becomes a Banach space. The following lemma can be found in Bahrouni and R\u{a}dulescu \cite[Lemma 3.1]{Bahrouni-Radulescu-2018}.

\begin{lemma}
    Suppose that $\Omega\subset \R^N$ is a bounded open domain and assume \eqref{cond_pq}. Then $W$ is a separable, reflexive space.
\end{lemma}

The following theorem states the compactness of the embedding $W$ into a suitable variable Lebesgue space $\Lp{r(\cdot)}$. For the proof we refer to Kaufmann, Rossi and Vidal \cite[Theorem 1.1]{Kaufmann-Rossi-Vidal-2017}.

\begin{theorem}\label{inj}
    Let $\Omega \subset \R^{N}$ be a Lipschitz bounded domain and $s\in (0,1)$. Let $q(\cdot)$, $p(\cdot,\cdot)$ be continuous variable exponents satisfying \eqref{cond_pq} with $sp(x,y)<N$ for $(x,y) \in \close\times\close$. Moreover $q(x)>p(x,x)$ for $x\in \close$. Assume that $r\colon \overline{\Omega}\to (1,\infty)$ is a continuous function such that
    \begin{align*}
    p^*_{s}(x)=\frac{Np(x,x)}{N-sp(x,x)}>r(x)\geq r_{-}:=\min_{x\in\close}r(x)>1
    \quad\text{for all }x\in\close.
    \end{align*}
    Then there exists a constant $C=C(N,s,p,q,r,\Omega)$ such that for every $f\in W$, it holds
    \begin{align*}
    \|f\|_{L^{r(\cdot)}(\Omega)}\leq C\|f\|_{W}.
    \end{align*}
    Thus, the space $W$ is continuously embedded in $L^{r(\cdot)}(\Omega)$ for any $r\in (1,p^*_s)$. Moreover, this embedding is compact.
\end{theorem}

We also refer to a similar result for traces for fractional Sobolev spaces with variable exponents, see Del Pezzo and Rossi \cite[Theorem 1.1]{Del-Pezzo-Rossi-2017}.

Under the assumption \eqref{cond_pq}, let $\mathcal{L}\colon W\to W^*$ be the nonlinear map defined by
\begin{align}\label{operator_fractional}
    \l\lan \mathcal{L} (u),\ph\r\ran= \int_{\Omega}\frac{|u(x)-u(y)|^{p(x,y)-2}(u(x)-u(y))(\ph(x)-\ph(y))}{|x-y|^{N+sp(x,y)}}\,dy.
\end{align}
It can be seen as the generalization of the fractional $p$-Laplacian in the constant exponent case and it is called fractional $p(\cdot,\cdot)$-Laplacian, denoted by $\mathcal{L}:=(-\Delta)^{s}_{p(\cdot,\cdot)}$.
Bahrouni and R\u{a}dulescu \cite[Lemma 4.2]{Bahrouni-Radulescu-2018} proved several properties of $\mathcal{L}$ which are stated in the next lemma.

\begin{lemma}$\ $
    \begin{enumerate}
    \item[(i)]
        $\mathcal{L}$ is a bounded and strictly monotone operator;
    \item[(ii)]
        $\mathcal{L}$ fulfills the $(\Ss_+$)-property, that is, $u_n \weak u$ in $W$ and $\limsup_{n\to \infty} \langle \mathcal{L}(u_n),u_n-u\rangle \leq 0$ imply $u_n\to u$  in $W$;
    \item[(iii)]
        $\mathcal{L}\colon W\to W^*$ is a homeomorphism.
    \end{enumerate}
\end{lemma}

The operator \eqref{operator_fractional} is related to the energy functional $J\colon W\to R$ defined by
\begin{align*}
    J(u)= \into\into \frac{|u(x)-u(y)|^{p(x,y)}}{p(x,y)|x-y|^{N+sp(x,y)}} \,dx\,dy\quad\text{for all }u \in W.
\end{align*}
It is clear that $J$ is well-defined on $W$ and $J\in C^1(W;\R)$ with the derivative given by
\begin{align*}
    \l\lan J'(u),\ph\r\ran = \l\lan \mathcal{L}(u),\ph\r\ran\quad\text{for all }u,\ph\in W,
\end{align*}
see Bahrouni and R\u{a}dulescu \cite[Lemma 4.1]{Bahrouni-Radulescu-2018}.

\begin{remark}\label{remark_section_2}
    Note that Theorem \ref{inj} remains true when $q(x)\geq p(x,x)$ for all $x\in\overline{\Omega}$, see Zhang and Zhang \cite{Zhang-Zhang-2020}. In existing articles working on $W$, see Bahrouni and R\u{a}dulescu \cite[Theorem 5.1]{Bahrouni-Radulescu-2018} or Kaufmann, Rossi and Vidal \cite[Theorem 1.4]{Kaufmann-Rossi-Vidal-2017}, the function $q$ is actually assumed to satisfy $q(x)>p(x,x)$ for all $x\in \overline\Omega$ due to some technical reason. Such spaces are actually not a
generalization of the fractional Sobolev space $W^{s,p}(\Omega)$.
\end{remark}

\section{Functional setting}\label{section_3}

The aim of this section is to give the basic properties of the
fractional $p(\cdot,\cdot)$-Laplacian with associated $p(\cdot,\cdot)$-Neumann boundary condition. After this, we are able to introduce the definition of a weak solution for the new Robin problem with $p(\cdot,\cdot)$-Neumann boundary condition stated in \eqref{eq}. In order to do this, we use some ideas developed by
Bahrouni, R\u{a}dulescu and Winkert \cite{Bahrouni-Radulescu-Winkert-2019} and Dipierro, Ros-Oton and Valdinoci \cite{Dipierro-Ros-Oton-Valdinoci-2017}.

We suppose the following assumptions:
\begin{enumerate}
    \item[(S)]
    $s\in \R$ with $s \in (0,1)$;
    \item[(P)]
    $p\colon \R^{2N} \to (1,+\infty)$ is a symmetric, continuous function bounded away from $1$, that is,
    \begin{align*}
        p(x,y)=p(y,x)\quad\text{for all }x,y\in \R^{2N}
    \end{align*}
    with
    \begin{align*}
        1<p^-&:= \min_{(x,y) \in \R^{2N}}p(x,y)\leq p(x,y)
        \leq p^+:=\max_{(x,y) \in \R^{2N}} p(x,y).
    \end{align*}
    and $sp^+<N$;
    \item[(G)]
    $g\in L^{1}(\R^{N}\setminus \Omega)$;
    \item[($\beta$)]
    $\beta \in L^{\infty}(\R^{N}\setminus \Omega)$ and $\beta \geq 0$ in $\R^{N}\setminus \Omega$;
\end{enumerate}

Let $u \colon \R^N\to\R$ be a measurable function and let $\overline{p}(x)=p(x,x)$ for all $x \in \R^{2N}$. We set
\begin{align*}
    \|u\|_{X}:=
    [u]_{s,\overline{p}(\cdot),\R^{2N}\setminus (\mathcal{C}\Omega)^2}
    +\|u\|_{L^{\overline{p}(\cdot)}(\Omega)}+\l\|\l|g\r|^{\frac{1}{\overline{p}(\cdot)}}u\r\|_{L^{\overline{p}(\cdot)}(\mathcal{C}\Omega)}
    +\l\|\beta^{\frac{1}{\overline{p}(\cdot)}}u\r\|_{L^{\overline{p}(\cdot)}(\mathcal{C}\Omega)},
\end{align*}
where $\mathcal{C}\Omega=\R^N\setminus\Omega$ and
\begin{align*}
    X:=\l\{u\colon\R^N\to \R \text{ measurable } : \  \|u\|_{X}<\infty\r\}.
\end{align*}

\begin{proposition}
    Let hypotheses (S), (P), (G) and ($\beta$) be satisfied. Then, $(X,\|\cdot\|_{X})$ is a reflexive Banach space.
\end{proposition}
\begin{proof}
    {\bf Step 1:} $(X,\|\cdot\|_{X})$ is a Banach space.

    It is easy to check that $\|\cdot\|_{X}$ is a norm on $X$. We only show that if $\|u\|_{X}=0$, then $u=0$ a.\,e.\,in $\R^N$. Indeed, from $\|u\|_{X}=0$, we get $\|u\|_{L^{\overline{p}(\cdot)}(\Omega)}=0$, which implies that
    \begin{align}\label{1}
    u=0 \quad \text{a.\,e.\,}\text{in } \Omega,
    \end{align}
    and
    \begin{align}\label{2}
    \int_{\R^{2N}\setminus (\mathcal{C}\Omega)^2}\frac{|u(x)-u(y)|^{p(x,y)}}{|x-y|^{N+sp(x,y)}}\,dx\,dy=0.
    \end{align}
    By \eqref{2}, we deduce that $u(x)=u(y)$ for a.\,a.\,$(x,y)\in\R^{2N}\setminus (\mathcal{C}\Omega)^2$, that is, $u=c\in\R$ a.\,e.\, in $\R^N$. By \eqref{1}, it easily follows that $c=0$, so $u=0$ a.\,e.\,in
    $\R^N$.

    Now, we prove that $X$ is complete. To this end, let $(u_k)_{k\in\mathbb{N}}$ be a Cauchy sequence in $X$. In particular, $(u_k)_{k\in\mathbb{N}}$ is a Cauchy sequence in $L^{\overline{p}(\cdot)}(\Omega)$ and so, up to a subsequence, there exists $u\in L^{\overline{p}(\cdot)}(\Omega)$ such that
    \begin{align*}
    u_k\to u \quad\text{in } \Lp{\overline{p}(\cdot)} \text{ and a.\,e.\,in } \Omega.
    \end{align*}
    Precisely, we find $Z_1\subset\R^N$ such that
    \begin{align}\label{3}
    |Z_1|=0\quad \text{and}\quad u_k(x)\to u(x)\quad \text{for every } x\in \Omega\setminus Z_1.
    \end{align}
    For any $U\colon\R^N\to\R$ and for any $(x,y)\in\R^{2N}$ we set
    \begin{align*}
    E_U(x,y):=\frac{(U(x)-U(y))\chi_{\R^{2N}\setminus (\mathcal{C}\Omega)^2}(x,y)}{|x-y|^{\frac{N}{p(x,y)}+s}},
    \end{align*}
    which implies
    \begin{align*}
    E_{u_k}(x,y)-E_{u_h}(x,y)=\frac{(u_k(x)-u_h(x)-u_k(y)+u_h(y))\chi_{\R^{2N}\setminus (\mathcal{C}\Omega)^2}(x,y)}{|x-y|^{\frac{N}{p(x,y)}+s}}.
    \end{align*}
    Using the fact that $(u_k)_{k\in\mathbb{N}}$ is a Cauchy sequence in $X$ and Proposition \ref{pprp1}, for every $\eps>0$ there exists $N_{\eps}\in \mathbb{N}$ such that for $h,k\geq N_{\eps}$, we have
    \begin{align*}
    \eps^{\frac{p^-}{p^+}}
    & \geq [E_{u_k}-E_{u_h}]^{\frac{p^-}{p^+}}_{s,p(\cdot,\cdot),\R^{2N}\setminus (\mathcal{C}\Omega)^2}\\
    &\geq \l (\int_{\R^{2N}\setminus (\mathcal{C}\Omega)^2} \frac{|(u_k-u_h)(x)-(u_k-u_k)(y)|^{p(x,y)}}{|x-y|^{N+sp(x,y)}}\,dx\,dy\r)^{p^-}\\
    &\geq \l\|E_{u_k}-E_{u_h}\r\|_{L^{p(\cdot,\cdot)}\l(\R^{2N}\r)}.
    \end{align*}
    Thus, $(E_{u_k})_{k\in\mathbb{N}}$ is a Cauchy sequence in $L^{p(\cdot,\cdot)}(\R^{2N})$ and so, up to a subsequence,
    we are able to assume that $E_{u_k}$ converges to some $E_u$ in  $L^{p(\cdot,\cdot)}(\R^{2N})$ and a.\,e.\,in $\R^{2N}$. This means we can find $Z_2\subset\R^{2N}$ such that
    \begin{align}\label{4}
    |Z_2|=0\ \ \text{and}\ \ E_{u_k}(x,y)\to E_u(x,y)\ \text{for every}\ (x,y)\in\R^{2N}\setminus Z_2.
    \end{align}
    For any $x\in \Omega$, we set
    \begin{align*}
    S_x&:=\l\{y\in\R^N:\ (x,y)\in\R^{2N}\setminus Z_2\r\},\\
    W&:=\l\{(x,y)\in\R^{2N}:\ x\in\Omega \text{ and }y\in\R^N\setminus S_x\r\},\\
    V& :=\l\{x\in\Omega:\ |\R^N\setminus S_x|=0\r\}.
    \end{align*}
    Proceeding exactly as in Dipierro, Ros-Oton and Valdinoci \cite[Proposition 3.1]{Dipierro-Ros-Oton-Valdinoci-2017} and
    Mugnai  and  Lippi \cite[Proposition 2.2]{Mugnai-Proietti-Lippi-2019} we get
    \begin{align*}
    |\Omega\setminus(V\setminus Z_1)|=|(\Omega\setminus V)\cup Z_1)|\leq |\Omega\setminus V|+|Z_1|=0.
    \end{align*}
    In particular $V\setminus Z_1\neq\emptyset$, so we can fix $x_0\in V\setminus Z_1$.

    Because of $x_0 \in \Omega \setminus Z_1$, from \eqref{3}, it follows
    \begin{align*}
    \lim_{k\to\infty}u_k(x_0)=u(x_0).
    \end{align*}
    In addition, since $x_0\in V$, we obtain $|\R^N\setminus S_{x_0}|=0$. Then, for a.\,a.\,$y\in \R^N$, this yields $(x_0,y) \in\R^{2N}\setminus Z_2$ and hence, by \eqref{4},
    \begin{align*}
    \lim_{k\to\infty}E_{u_k}(x_0,y)=E_u(x_0,y).
    \end{align*}
    Since $\Omega\times (\mathcal{C}\Omega)\subseteq \R^{2N}\setminus(\mathcal{C}\Omega)^2$, we have
    \begin{align*}
    E_{u_k}(x_0,y):=\frac{u_k(x_0)-u_k(y)}{|x_0-y|^{s+\frac{N}{p(x,y)}}} \quad  \text{for a.\,a.\, } y\in \mathcal{C}\Omega.
    \end{align*}
    But this implies
    \begin{align*}
    \lim_{k\to\infty}u_k(y)
    & =\lim_{k\to\infty} \l(u_k(x_0)-|x_0-y|^{\frac{N}{p(x,y)}+s}E_{u_k}(x_0,y)\r)\\
    &= u(x_0)-|x_0-y|^{\frac{N}{p(x,y)}+s}E_{u}(x_0,y)\quad \text{for a.\,a.\, } y\in \mathcal{C}\Omega.
    \end{align*}
    Combining this with \eqref{3} we see that $u_k$ converges a.\,e.\,in $\R^N$ to some $u$ in $\R^N$. Since $u_k$ is a Cauchy sequence in $X$, for any $\eps>0$ there exists $N_{\eps}>0$ such that, for any $h\geq N_{\eps}$, we have by applying Fatou's Lemma,
    \begin{align*}
    \eps
    & \geq \liminf_{k\to+\infty}\|u_h-u_k\|_{X}\\
    & \geq \liminf_{k\to+\infty}\l(\rho_{s,p(\cdot,\cdot),\R^{2N}\setminus(\mathcal{C}\Omega)^2}(u_h-u_k)\r)^{\frac{1}{p^-}}
    +\liminf_{k\to+\infty}\l(\int_{\Omega}|u_h-u_k|^{\overline{p}(x)}\,dx\r)^{\frac{1}{p^-}}\\
    &\quad +\liminf_{k\to+\infty}\l(\int_{\mathcal{C}\Omega}|g||u_h-u_k|^{\overline{p}(x)}\,dx\r)^{\frac{1}{p^-}}
    +\liminf_{k\to+\infty}\l(\int_{\mathcal{C}\Omega}|\beta||u_h-u_k|^{\overline{p}(x)}\,dx\r)^{\frac{1}{p^-}}\\
    & \geq \l(\rho_{s,p(\cdot,\cdot),\R^{2N}\setminus(\mathcal{C}\Omega)^2}(u_h-u)\r)^{\frac{1}{p^-}}
    +\l(\int_{\Omega}|u_h-u|^{\overline{p}(x)}\,dx\r)^{\frac{1}{p^-}}\\
    &\quad +\l(\int_{\mathcal{C}\Omega}|g||u_h-u|^{\overline{p}(x)}\,dx\r)^{\frac{1}{p^-}}
    +\l(\int_{\mathcal{C}\Omega}|\beta||u_h-u|^{\overline{p}(x)}\,dx\r)^{\frac{1}{p^-}}\\
    &\geq [u_h-u]_{s,p,\R^{2N}\setminus (\mathcal{C}\Omega)^2}^{\frac{p^+}{p^-}}
    +\|u_h-u\|_{L^{\overline{p}(\cdot)}(\Omega)}^{\frac{p^+}{p^-}}
    + \||g|^{\frac{1}{\overline{p}(\cdot)}}(u_h-u)\|_{L^{\overline{p}(\cdot)}(\mathcal{C}\Omega)}^{\frac{p^+}{p^-}}\\
    & \quad +\l(\int_{\mathcal{C}\Omega}|\beta||u_h-u|^{\overline{p}(x)}\,dx\r)^{\frac{p^{+}}{p^-}}\\
    &\geq \frac{1}{4^{\frac{p^{+}}{p^-}-1}}\|u_h-u\|_{X}^{\frac{p^{+}}{p^{-}}}.
    \end{align*}
    Therefore, $u_h$ converges to $u$ in $X$ and so, $X$ is complete.

    {\bf Step 2:} $X$ is a reflexive space.

    Consider the space
    \begin{align*}
    Y=L^{\overline{p}(x)}(\Omega)
    \times L^{\overline{p}(x)}(\mathcal{C}\Omega)
    \times L^{\overline{p}(x)}(\mathcal{C}\Omega)
    \times L^{p(x,y)}\l(\R^{2N}\setminus(\mathcal{C}\Omega)^2\r).
    \end{align*}
    endowed with the norm
    \begin{align*}
    \|v\|_{Y}:=
    \|v\|_{L^{\overline{p}(\cdot)}(\Omega)}
    +\|\beta^{\frac{1}{\overline{p}(\cdot)}} v\|_{L^{\overline{p}(\cdot)}(\mathcal{C}\Omega)}
    +\||g|^{\frac{1}{\overline{p}(\cdot)}} v\|_{L^{\overline{p}(\cdot)}(\mathcal{C}\Omega)}
    +[v]_{s,p(\cdot,\cdot),\R^{2N}\setminus (\mathcal{C}\Omega)^2}.
    \end{align*}
    We notice that $(Y,\|.\|_{Y})$ is a reflexive Banach space. We consider the map $T\colon X\to Y$ defined as
    \begin{align*}
    T(u):=\l(u,\beta^{\frac{1}{\overline{p}(\cdot)}} u,g^{\frac{1}{\overline{p}(x)}} u,\frac{|u(x)-u(y)|}{|x-y|^{\frac{N}{p(x,y)}+s}}\r).
    \end{align*}
    By construction, we have that
    \begin{align*}
    \|T(u)\|_{Y}=\|u\|_{X}.
    \end{align*}
    Hence, $T$ is an isometry from $X$ to the reflexive space $Y$. This shows that $X$ is reflexive.
\end{proof}

\begin{proposition}\label{compact}
    Let hypotheses (S), (P), (G) and ($\beta$) be satisfied. Then, for any $r\in C_{+}(\overline{\Omega}) $ with $1<r(x)< p_{s}^{*}(x)$ for all $x\in \overline{\Omega}$, there exists a constant $\alpha>0$ such that
    \begin{align*}
    \|u\|_{L^{r(\cdot)}(\Omega)}\leq \alpha\|u\|_{X}\quad \text{for all }u\in X.
    \end{align*}
    Moreover, this embedding is compact.
\end{proposition}

\begin{proof}
    By the assumptions it is clear that
    \begin{align*}
    \|u\|_{E}\leq \|u\|_{X} \quad \text{for all }u\in X.
    \end{align*}
    Therefore, using Theorem \ref{inj} and Remark \ref{remark_section_2}, we get our desired result.
\end{proof}

Note that the norm $\|\cdot\|_{X}$ is equivalent on $X$ to
\begin{align}\label{equivalent_norm}
    \begin{split}
    & \|u\|\\
    &=\inf\left\{\mu\geq 0 \ \bigg| \ \rho\left(\frac{u}{\mu}\right)\leq 1\right\}\\
    &=\inf\left\{\mu\geq 0 \ \bigg | \  \int_{\R^{2N}\setminus (\mathcal{C}\Omega)^2}\frac{|u(x)-u(u)|^{p(x,y)}}{\mu^{p(x,y)}p(x,y)(|x-y|)^{N+sp(x,y)}}\,dx\,dy
    + \int_{\Omega}\frac{|u|^{\overline{p}(x)}}{\overline{p}(x)\mu^{\overline{p}(x)}}\,dx\right.\\
    &\qquad \qquad \qquad \quad \left.+ \int_{\mathcal{C}\Omega}\frac{g(x)}{\mu^{\overline{p}(x)}\overline{p}(x)} |u|^{\overline{p}(x)}\,dx
    + \int_{\mathcal{C}\Omega}\frac{\beta(x)}{\mu^{\overline{p}(x)}\overline{p}(x)}
    |u|^{\overline{p}(x)}\,dx\leq 1\right\},
    \end{split}
\end{align}
where the modular $\rho\colon X\to \R$ is defined by
\begin{align*}
    \rho\left(u\right)
    &=  \int_{\R^{2N}\setminus (\mathcal{C}\Omega)^2}\frac{|u(x)-u(u)|^{p(x,y)}}{p(x,y)(|x-y|)^{N+sp(x,y)}}\,dx\,dy
    + \int_{\Omega}\frac{|u|^{\overline{p}(x)}}{\overline{p}(x)}\,dx\\
    &\quad + \int_{\mathcal{C}\Omega}\frac{g(x)}{\overline{p}(x)} |u|^{\overline{p}(x)}\,dx
    + \int_{\mathcal{C}\Omega}\frac{\beta(x)}{\overline{p}(x)}
    |u|^{\overline{p}(x)}\,dx.
\end{align*}

The following lemma will be helpful in later considerations.

\begin{lemma}\label{modular}
    Let hypotheses (S), (P), (G) and ($\beta$) be satisfied and let $u \in X$. Then the following hold:
    \begin{enumerate}
    \item[(i)]
        For $u\neq 0$ we have: $\|u\|_{X}=a$ if and only if $\rho(\frac{u}{a})=1$;
    \item[(ii)]
        $\|u\|_{X}<1$ implies $\frac{\|u\|_{X}^{p{+}}}{4^{p^{+}-1}}\leq \rho(u)\leq 4\|u\|_{X}^{p^{-}}$;
    \item[(iii)]
        $\|u\|_{X}>1$ implies $\|u\|_{X}^{p^{-}}\leq \rho(u)$.
    \end{enumerate}
\end{lemma}

\begin{proof}
    (i) It is clear that the mapping $\lambda \mapsto \rho(\lambda u)$ is a continuous, convex, even function, which is strictly increasing on $ [0,+\infty)$. Thus, by the definition of $\rho$ and the equivalent norm given in \eqref{equivalent_norm}, we have
    \begin{align*}
    \|u\|_{X}=a \quad \Longleftrightarrow \quad \rho\left(\frac{u}{a}\right)=1.
    \end{align*}

    (ii) Let $u\in X$ be such that $\|u\|_X<1$, then
    \begin{align*}
    & [u]_{s,p(\cdot,\cdot),\R^{2N}\setminus (\mathcal{C}\Omega)^2}<1, && \|u\|_{L^{\overline{p}(\cdot)}(\Omega)}<1,\\[1ex]
    &\l\|\l|g\r|^{\frac{1}{\overline{p}(\cdot)}}u\r\|_{L^{\overline{p}(\cdot)}(\mathcal{C}\Omega)}<1,&& \l\|\beta^{\frac{1}{\overline{p}(\cdot)}} u\r\|_{L^{\overline{p}(\cdot)}(\mathcal{C}\Omega)}<1.
    \end{align*}
    By the convexity of $\rho$ along with Proposition \ref{pprp1} we obtain the assertion.

    (iii) Let $u\in X$ be such that $\|u\|_{X}>1$. From (i) it follows
    \begin{align*}
    \rho\left(\frac{u}{\|u\|_X}\right)
    &= \int_{\R^{2N}\setminus (\mathcal{C}\Omega)^2}\frac{|u(x)-u(u)|^{p(x,y)}}{p(x,y)\|u\|_{X}^{p(x,y)}(|x-y|)^{N+sp(x,y)}}\,dx\,dy
    + \int_{\Omega}\frac{|u|^{\overline{p}(x)}}{\overline{p}(x)\|u\|_{X}^{\overline{p}(x)}}\,dx\\
    &\quad + \int_{\mathcal{C}\Omega}\frac{g(x)}{\overline{p}(x)\|u\|_{X}^{\overline{p}(x)}} |u|^{\overline{p}(x)}\,dx
    + \int_{\mathcal{C}\Omega}\frac{\beta(x)}{\overline{p}(x)\|u\|_{X}^{\overline{p}(x)}}|u|^{\overline{p}(x)}\,dx =1.
    \end{align*}
    Then, by the mean value theorem, there exist $(x_{1},y_{1})\in \R^{2N}\setminus (\mathcal{C}\Omega)^2$, $x_{2}\in \overline{\Omega}$, $x_{3},x_{4}\in \mathcal{C}\Omega $ such that
    \begin{align*}
    1
    &=\frac{1}{\|u\|_{X}^{p(x_{1},y_{1})}}\int_{\R^{2N}\setminus (\mathcal{C}\Omega)^2}\frac{|u(x)-u(u)|^{p(x,y)}}{p(x,y)(|x-y|)^{N+sp(x,y)}}\,dx\,dy
    +\frac{1}{\|u\|_{X}^{\overline{p}(x_2)}}
    \int_{\Omega}\frac{|u|^{\overline{p}(x)}}{\overline{p}(x)}\,dx\\
    & \quad +\frac{1}{\|u\|_{X}^{\overline{p}(x_3)}} \int_{\mathcal{C}\Omega}\frac{g(x)}{\overline{p}(x)} |u|^{\overline{p}(x)}\,
    dx+\frac{1}{\|u\|_{X}^{\overline{p}(x_4)}} \int_{\mathcal{C}\Omega}\frac{\beta(x)}{\overline{p}(x)}
    |u|^{\overline{p}(x)}\,dx.
    \end{align*}
    Since $\|u\|_{X}>1$, it follows that
    \begin{align*}
    1
    &\leq\frac{1}{\|u\|_{X}^{p^{-}}} \left[\int_{\R^{2N}\setminus (\mathcal{C}\Omega)^2}\frac{|u(x)-u(u)|^{p(x,y)}}{p(x,y)(|x-y|)^{N+sp(x,y)}}\,dx\,dy
    +\int_{\Omega}\frac{|u|^{\overline{p}(x)}}{\overline{p}(x)}\,dx\right]\\
    &\quad +\frac{1}{\|u\|_{X}^{p^{-}}} \left[  \int_{\mathcal{C}\Omega}\frac{\beta(x)}{\overline{p}(x)}
    |u|^{\overline{p}(x)}\,dx+ \int_{\mathcal{C}\Omega}\frac{g(x)}{\overline{p}(x)} |u|^{\overline{p}(x)}\,dx\right].
    \end{align*}
    This finishes the proof.
\end{proof}

\begin{lemma}\label{ss}
    Let hypotheses (S), (P), (G) and ($\beta$) be satisfied. Then $\rho\colon X\to \R$ and $\rho'\colon X\to X^*$ have the following properties:
    \begin{enumerate}
    \item[(i)]
        The function $\rho$ is of class $C^{1}(X,\R)$ and $\rho'\colon X \to X^*$ is coercive, that is,
        \begin{align*}
        \frac{\langle\rho'(u),u\rangle_{X}}{\|u\|_X}\to +\infty \quad \text{as }\|u\|_X\to +\infty;
        \end{align*}
    \item[(ii)]
        $\rho'$ is strictly monotone operator.
    \item[(iii)]
        $\rho'$ is a mapping of type $(\Ss_+$), that is, if $u_{n}\rightharpoonup u$ in $X$ and $ \limsup_{n\to
        +\infty}\,\langle \rho'(u_{n}),u_{n}-u\rangle_{X} \leq 0$, then $u_{n}\to u$ in $X$.
    \end{enumerate}
\end{lemma}

\begin{proof}
    (i) Evidently, from the definition of $\rho$, we conclude that $\rho \in C^{1}(X,\R)$. By Lemma \ref{modular}, for $\|u\|_{X}>1$, we obtain
    \begin{align*}
    \langle\rho'(u),u\rangle_{X}\geq \rho(u)\geq \|u\|_X^{p^{-}}.
    \end{align*}
    Then
    \begin{align*}
    \frac{\langle\rho'(u),u\rangle_{X}}{\|u\|_X}\geq \|u\|_X^{p^{-}-1} \to +\infty
    \end{align*}
    as $\|u\|_{X}\to +\infty$ since $p^{-}>1 $.

    (ii) The strict monotonicity of $\rho'$ is a direct consequence of the well-known Simon inequalities
    \begin{align*}
    &\left|x-y\right|^{p}\leq  c_p\left(\left|x\right|^{p-2}x-\left|y\right|^{p-2}y\right)\cdot(x-y) \quad\text{if }p \geq 2,
    \end{align*}
    and
    \begin{align*}
    \begin{split}
        \left|x-y\right|^{p} \leq C_p & \left[\left(\left| x\right|^{p-2}x-\left|y\right|^{p-2}y\right)\cdot(x-y)\right]^\frac{p}{2}\\
        & \quad \times \left(\left|x\right|^p+\left|y\right|^p\right)^{\frac{2-p}{p}} \quad \text{if }p\in (1,2),
    \end{split}
    \end{align*}
    for all $x,y \in \R^N$, where $c_p$ and $C_p$ are positive constants depending only on $p$,
    see Lindqvist \cite[p.\,71]{Lindqvist-2006} or Filippucci, Pucci and R\u adulescu \cite[p. 713]{Filippucci-Pucci-Radulescu-2008}.

    (iii) By applying (i) and (ii), the proof of assertion (iii) is identical to the proof of Theorem  3.1 in Bahrouni and R\u{a}dulescu \cite{Bahrouni-Radulescu-2018}.
\end{proof}

Now we are interested in a nonlocal analogue of the divergence theorem also known as integration by parts formula. We have the following result.

\begin{proposition}
      Let hypotheses (S), (P), (G) and ($\beta$) be satisfied and let $u$ be any bounded $C^2$-function in $\R^N$. Then,
      \begin{align*}
      \int_{\Omega}\l(-\Delta\r)^s_{p(\cdot,\cdot)}u(x)\,dx
      =-\int_{\R^N\setminus\Omega}\mathcal{N}_{s,p(\cdot,\cdot)}u(x)\,dx.
      \end{align*}
\end{proposition}

\begin{proof}
    From (P) we know that $p$ is symmetric. We obtain
    \begin{align*}
    &\int_{\Omega}\int_{\Omega}|u(x)-u(y)|^{p(x,y)-2}\frac{(u(x)-u(y))}{|x-y|^{N+sp(x,y)}}\,dx\,dy\\
    &=-\int_{\Omega}\int_{\Omega}|u(x)-u(y)|^{p(x,y)-2}\frac{(u(y)-u(x))}{|x-y|^{N+sp(x,y)}}\,dx\,dy
        =0.
    \end{align*}
    It follows that
    \begin{align*}
    \int_{\Omega}\l(-\Delta\r)^s_{p(\cdot,\cdot)}u(x)\,dx
    &=\int_{\Omega}\lim_{\eps\to0}\int_{\R^N\setminus B_{\eps}(x)}|u(x)-u(y)|^{p(x,y)-2}\frac{(u(x)-u(y))}{|x-y|^{N+sp(x,y)}}\,dy\,dx\\
    &=\int_{\Omega}\lim_{\eps\to0}\l[\int_{\R^N\setminus\Omega}|u(x)-u(y)|^{p(x,y)-2}\frac{(u(x)-u(y))}{|x-y|^{N+sp(x,y)}}\,dy\right.\\
    &\l.\qquad+\int_{\Omega\setminus B_{\eps}(x)} |u(x)-u(y)|^{p(x,y)-2}\frac{(u(x)-u(y))}{|x-y|^{N+sp(x,y)}}dy\r]\,dx\\
    &=\int_{\Omega}\int_{\R^N\setminus\Omega} |u(x)-u(y)|^{p(x,y)-2}\frac{(u(x)-u(y))}{|x-y|^{N+sp(x,y)}}\,dy\,dx\\
    &=\int_{\R^N\setminus\Omega}\int_{\Omega}|u(x)-u(y)|^{p(x,y)-2}\frac{(u(x)-u(y))}{|x-y|^{N+sp(x,y)}}\,dx\,dy\\
    &=-\int_{\R^N\setminus\Omega}\mathcal{N}_{s,p(\cdot,\cdot)}u(y)\,dy.
    \end{align*}
\end{proof}

\begin{proposition}
    Let hypotheses (S), (P), (G) and ($\beta$) be satisfied.
    Let $u$ and $v$ be bounded $C^2$-functions in $\R^N$. Then,
    \begin{align*}
    &\frac{1}{2}\int_{\R^{2N}\setminus(\mathcal{C}\Omega)^2}|u(x)-u(y)|^{p(x,y)-2}\frac{(u(x)-u(y))(v(x)-v(y))}{|x-y|^{N+sp(x,y)}}\,dx\,dy\\
    &=\int_{\Omega}v(-\Delta)^s_{p(\cdot,\cdot)}u\,dx
    +\int_{\mathcal{C}\Omega}v\mathcal{N}_{s,p(\cdot,\cdot)}\, dx.
    \end{align*}
\end{proposition}

\begin{proof}
    By symmetry, we have
    \begin{align*}
    & \frac{1}{2}\int_{\R^{2N}\setminus(\mathcal{C}\Omega)^2}|u(x)-u(y)|^{p(x,y)-2}\frac{(u(x)-u(y))(v(x)-v(y))}{|x-y|^{N+sp(x,y)}}\,dx\,dy\\
    &=\int_{\Omega}\int_{\R^{N}}v(x)|u(x)-u(y)|^{p(x,y)-2}\frac{(u(x)-u(y))}{|x-y|^{N+sp(x,y)}}\,dy\,dx\\
    &\quad +\int_{\mathcal{C}\Omega}\int_{\Omega}v(x)|u(x)-u(y)|^{p(x,y)-2}\frac{(u(x)-u(y))}{|x-y|^{N+sp(x,y)}}\,dy\,dx.
    \end{align*}
    Thus, using \eqref{operator} and \eqref{Neumann boundary}, the identity follows.
\end{proof}

Based on the integration by parts formula we are now in the position to state the natural definition of a weak solution for problem \eqref{eq}. First, to simplify the notation, for arbitrary functions $u,v\colon \R^N\to\R$, we set
\begin{align*}
    \mathcal{A}_{s,p}(u,v)
    &=\frac{1}{2}\int_{\R^{2N}\setminus(\mathcal{C}\Omega)^2}|u(x)-u(y)|^{p(x,y)-2}\frac{(u(x)-u(y))(v(x)-v(y))}{|x-y|^{N+sp(x,y)}}\,dx\,dy\\
    &\quad +\int_{\Omega}|u|^{\overline{p}(x)-2}uv\,dx+ \int_{\mathcal{C}\Omega}\beta(x)|u|^{\overline{p}(x)-2}uv\,dx.
\end{align*}
We say that $u\in X$ is a weak solution of \eqref{eq}, if
\begin{align}\label{weak_solution_Neumann}
      \mathcal{A}_{s,p}(u,v)
      =\int_{\Omega}f(x,u)v\, dx+\int_{\mathcal{C}\Omega}gv\, dx.
\end{align}
is satisfied for every $v \in X$. As a consequence of this definition, we have the following result.

\begin{proposition}
    Let hypotheses (S), (P), (G) and ($\beta$) be satisfied. and let u be a weak solution of \eqref{eq}. Then,
    \begin{align*}
    \mathcal{N}_{s,p(\cdot,\cdot)}u+\beta(x)|u|^{\overline{p}(x)-2}u =g  \quad \text{a.\,e.\,in }\R^N\setminus\overline{\Omega}.
    \end{align*}
\end{proposition}

\begin{proof}
    First, we take $v\in X$ such that $v\equiv0$ in $\Omega$ as a test function in \eqref{weak_solution_Neumann}. Then
    \begin{align*}
    &\int_{\mathcal{C}\Omega}gv\, dx\\
    & = \mathcal{A}_{s,p}(u,v)\\
    &= -\frac{1}{2}\int_{\Omega}\int_{\R^{N}\setminus\overline{\Omega}}\frac{|u(x)-u(y)|^{p(x,y)-2}(u(x)-u(y))v(y)}{|x-y|^{N+sp(x,y)}}\,dy\,dx\\
    &\qquad +\frac{1}{2}\int_{\R^{N}\setminus\overline{\Omega}}\int_{\Omega}\frac{|u(x)-u(y)|^{p(x,y)-2}(u(x)-u(y))v(x)}                                   {|x-y|^{N+sp(x,y)}}\,dy\,dx
    + \int_{\mathcal{C}\Omega}\beta(x)|u|^{\overline{p}(x)-2}uv\,dx\\
    &=\int_{\Omega}\int_{\R^{N}\setminus\overline{\Omega}}\frac{|u(x)-u(y)|^{p(x,y)-2}(u(x)-u(y))v(y)} {|x-y|^{N+sp(x,y)}}\,dy\,dx+ \int_{\mathcal{C}\Omega}\beta(x)|u|^{\overline{p}(x)-2}uv\,dx\\
    &=\int_{\R^{N}\setminus\overline{\Omega}}v(x)\int_{\Omega}\frac{|u(x)-u(y)|^{p(x,y)-2}(u(x)-u(y))} {|x-y|^{N+sp(x,y)}}\,dx\,dy+ \int_{\mathcal{C}\Omega}\beta(x)|u|^{\overline{p}(x)-2}uv\,dx\\
    &=\int_{\R^{N}\setminus\overline{\Omega}}v(y)\mathcal{N}_{s,p(\cdot,\cdot)}u(y)\,dy+ \int_{\mathcal{C}\Omega}\beta(x)|u|^{\overline{p}(x)-2}uv\,dx.
    \end{align*}
    Therefore,
    \begin{align*}
    \int_{\R^{N}\setminus\overline{\Omega}}\Big(\mathcal{N}_{s,p(\cdot,\cdot)}u(x)+\beta(x)|u|^{\overline{p}(x)-2}u-g(x)\Big)v(x)\,dx=0
    \end{align*}
    for every $v\in X$ which is $0$ in $\Omega$. In particular, this is true for every  $v\in C^{\infty}_{c}(\R^{N}\setminus\overline{\Omega})$, and so
    \begin{align*}
    \mathcal{N}_{s,p(\cdot,\cdot)}u(x)
    +\beta(x)|u|^{\overline{p}(x)-2}u
    =  g(x) \quad \text{a.\,e.\,in } \R^N\setminus\overline{\Omega}.
    \end{align*}
\end{proof}

\begin{proposition}\label{prp4}
    Let hypotheses (S), (P), (G) and ($\beta$) be satisfied. Let $I\colon X\to\R$ be the functional defined by
    \begin{align*}
    I(u)&=
    \int_{\R^{2N}\setminus(\mathcal{C}\Omega)^2}\frac{|u(x)-u(y)|^{p(x,y)}}{2p(x,y)|x-y|^{N+sp(x,y)}}\,dx\,dy
    +\int_{\mathcal{C}\Omega}\frac{\beta(x)|u|^{\overline{p}(x)}v}{\overline{p}(x)}\,dx\\
    &\quad -\int_{\Omega}f(x,u)u\,dx-\int_{\mathcal{C}\Omega}gu\,dx
    \quad \text{ for every }u \in X.
    \end{align*}
    Then any critical point of $I$ is a weak solution of problem \eqref{eq}.
\end{proposition}

\begin{proof}
      We only show that $\mathrm{I}$ is well defined on $X$. The rest follows by standard argument.

      Applying H\"older's inequality and condition (F), we have
      \begin{align}\label{fc}
      \begin{split}
          \int_{\Omega}f(x,u)u\,dx
          & \leq \int_{\Omega} b(x)|u|^{q(x)-1}u\,dx
          \leq \int_{\Omega} b(x)|u|^{q(x)}\,dx\\
          &\leq c \|b\|_{r(\cdot)}\l\||u|^{q(\cdot)}\r\|_{r'(\cdot)}<\infty.
      \end{split}
      \end{align}
      Again, by Proposition \ref{pprp1} and condition (G), we infer that
      \begin{align}\label{gc}
      \begin{split}
          \int_{\mathcal{C}\Omega}gu\,dx
          &\leq\int_{\mathcal{C}\Omega}|g|^{\frac{1}{\overline{p}'(x)}}|g|^{\frac{1}{\overline{p}(x)}} |u|\,dx
          \leq 2\|g\|_{L^1(\mathcal{C}\Omega)}\l \|\l|g\r|^{\frac{1}{\overline{p}(\cdot)}}u\r\|_{L^{\overline{p}(\cdot)}(\mathcal{C}\Omega)}
          \leq C\|u\|_{X}.
    \end{split}
    \end{align}
    Combining \eqref{fc} and \eqref{gc}, we conclude that
    $\mathrm{I}$ is well defined.
\end{proof}

\section{Existence results for fractional Robin problems with variable exponent}\label{section_4}

In this section we suppose conditions (S), (P), ($\beta$) and
\begin{enumerate}
    \item[(F)]
    Let $g \equiv 0$ and let $f\colon \Omega \times \R \to \R$ be a Carath\'eodory function given by
    \begin{align*}
        f(x,u)=\lambda V(x)|u|^{q(x)-2}u \quad \text{for all } x\in \Omega,
    \end{align*}
    where $q\in C_{+}(\Omega)$ such that $1<q(x)<\overline{p}(x)$ in $\close$ and with
    \begin{align*}
        V\in L^{r(\cdot)} \text{ such that }r\in C_{+}(\Omega) \text{ and } 1<r'(x)q(x)<p^*_{s}(x) \text{ for all }x\in \close.
    \end{align*}
    Moreover, we suppose that there exists a nonempty subset $\Omega_{0}\subset \Omega$ such that
    \begin{align*}
        V(x)>0 \quad \text{for all }x\in \Omega_0.
    \end{align*}
\end{enumerate}

The aim of this section is to prove the existence of at least one weak solution of \eqref{eq} when the parameter $\lambda>0$ is small enough. The proof is based in the results of the previous section in combination with variational methods.

First we introduce the variational setting for problem \eqref{eq}. To this end, we denote by $I\colon X\to \R$ the energy function of
problem \eqref{eq} which is given by
\begin{align*}
    I(u)
    & =\int_{\R^{2N}\setminus(\mathcal{C}\Omega)^2}\frac{|u(x)-u(y)|^{p(x,y)}}{2p(x,y)|x-y|^{N+sp(x,y)}}\,dx\,dy
    + \int_{\Omega}\frac{|u|^{\overline{p}(x)}}{\overline{p}(x)}\,dx\\
    &\quad +\int_{\mathcal{C}\Omega} \frac{\beta(x)|u|^{\overline{p}(x)}v}{\overline{p}(x)}\,dx
    -\lambda\int_{\Omega}\frac{V(x)}{q(x)}|u|^{q(x)}\,dx.
\end{align*}

Note that under the assumptions (S), (P), ($\beta$) and (F) along with Proposition \ref{prp4} it is easy to see that the functional $I$ is well-defined, of class $C^{1}$ on $X$ and any critical point of $I$is a weak solution of problem \eqref{eq}.

We start with two auxiliary results.

\begin{lemma}\label{pos}
    Let hypotheses (S), (P), (F) and ($\beta$) be satisfied. Then there is $\lambda^*>0$ such that for any $\lambda \in (0,\lambda^*)$ there exist $\rho>0$ and $a>0$ such that
    \begin{align*}
    I(u)\geq a>0 \quad \text{for any } u\in X \text{ with }
    \|u\|=\rho.
    \end{align*}
\end{lemma}

\begin{proof}
    From Proposition \ref{compact} we have
    \begin{align}\label{g11}
    \|u\|_{r'(\cdot)q(\cdot)}\leq \alpha \|u\|_X \quad \text{for all }u\in X.
    \end{align}
    Fix $\rho \in \l(0, \min\l(1,\frac{1}{\alpha}\r)\r)$. Then inequality \eqref{g11} implies that
    \begin{align*}
    \|u\|_{r'(x)q(x)}<1 \quad \text{for all }u\in X \text{ with } \|u\|_{X}=\rho.
    \end{align*}
    Thus, by applying H\"older's inequality and Proposition \ref{pprp1}, we get
    \begin{align}\label{g12}
    \begin{split}
        \int_{\Omega}V(x)|u|^{q(x)}\,dx
        &\leq 2 \|V\|_{r(\cdot)}\|u\|_{r'(\cdot)q(\cdot)}^{q^{-}}\\
        &\leq 2 \alpha^{q^{-}}\|V\|_{r(\cdot)}\|u\|_{X}^{q^{-}} \quad \text{for all }u\in X \text{ with }  \|u\|_{X}=\rho.
    \end{split}
    \end{align}
    Hence, using \eqref{g12} and Lemma \ref{modular}, we obtain for any $u\in X$ with $\|u\|_{X}=\rho$ that
    \begin{align*}
    I(u)
    &=\int_{\R^{2N}\setminus(\mathcal{C}\Omega)^2}\frac{|u(x)-u(y)|^{p(x,y)}}{p(x,y)|x-y|^{N+sp(x,y)}}\,dx\,dy +\int_{\mathcal{C}\Omega} \frac{\beta(x)|u|^{\overline{p}(x)}v}{\overline{p}(x)}\,dx
    +\int_{\Omega}\frac{|u|^{\overline{p}(x)}}{\overline{p}(x)}\,dx\\
    &\quad -\lambda \int_{\Omega}\frac{V(x)}{q(x)}|u|^{q(x)}\,dx\\
    &\geq \frac{1}{p^{+}}\l(\int_{\R^{2N}\setminus(\mathcal{C}\Omega)^2}\frac{|u(x)-u(y)|^{p(x,y)}}{|x-y|^{N+sp(x,y)}}\,dx\,dy +\int_{\mathcal{C}\Omega}\beta(x)|u|^{\overline{p}(x)}\,dx
    +\int_{\Omega}|u|^{\overline{p}(x)}\,dx\r)\\
    &\quad - \lambda\frac{1}{q^{-}}\int_{\Omega}V(x)|u|^{q(x)}\,dx\\
    &\geq\frac{1}{p^{+}3^{p^{+}-1}}\|u\|_{X}^{p^{+}}
    -\lambda\frac{2 \alpha^{q^{-}}\|V\|_{r(\cdot)}}{q^{-}}\|u\|_{X}^{q^{-}}\\
    &=\frac{1}{p^{+}3^{p^{+}-1}}\rho^{p^{+}}
    -\lambda\frac{2 \alpha^{q^{-}}\|V\|_{r(\cdot)}}{q^{-}}\rho^{q^{-}}\\
    &=\rho^{q^{-}}\l(\frac{1}{p^{+}3^{p^{+}-1}} \rho^{p^{+}-q^{-}}-\lambda\frac{2 \alpha^{q^{-}}\|V\|_{r(\cdot)}}{q^{-}}\r).
    \end{align*}
    We set
    \begin{align*}
    \lambda^*=
    \frac{q^{-}}{4p^{+}3^{p^{+}-1} \alpha^{q^{-}}\|V\|_{r(\cdot)}}.
    \end{align*}
    Then, combining this with the inequality above gives
    \begin{align*}
    I(u)\geq \frac{1}{2p^{+}3^{p^{+}-1}}=a>0 \quad \text{for all }u\in X \text{ with }\|u\|_X=\rho,
    \end{align*}
    where $\lambda \in (0,\lambda^*)$. This completes the proof.
\end{proof}

\begin{lemma}\label{neg}
    Let hypotheses (S), (P), (F) and ($\beta$) be satisfied. Then, there exists $\varphi \in X$ such that
    \begin{align*}
    I(t\varphi)<0 \quad \text{for } t>0 \text{ small enough}.
    \end{align*}
\end{lemma}

\begin{proof}
    We denote by
    \begin{align*}
    p_{0}^{-}=\inf_{x\in \overline{\Omega}_{0}}\overline{p}(x) \quad \text{and} \quad
    q_{0}^{-}=\inf_{x\in \overline{\Omega}_{0}}q(x).
    \end{align*}
    Then, from condition (F), there exist $\eps_{0}>0$ and an open set $\Omega_{1}\subset \Omega_{0}$ such that
    \begin{align*}
    q_{0}^{-}+\eps_{0}<p_{0}^{-}
    \quad \text{and} \quad
    |q(x)-q_{0}^{-}|<\eps_0 \quad \text{for all }x\in \Omega_1.
    \end{align*}
    Thus
    \begin{align}\label{g21}
    q(x)\leq q_{0}^{-}+\eps_{0}<p_{0}^{-} \quad \text{for all } x\in \Omega_1.
    \end{align}
    Let $\varphi \in C^{\infty}_{0}(\Omega_0)$ such that $\overline{\Omega}_{1}\subset \mbox{supp}(\varphi)$, $\varphi=1$ for all $x \in \overline{\Omega}_{1}$ and $ 0 \leq \varphi \leq 1$ in $\Omega_0$. Then, it follows, for $t\in (0,1)$ small enough by applying \eqref{g21}, that
    \begin{align*}
    I(t\varphi)
    &=\int_{\R^{2N}\setminus(\mathcal{C}\Omega)^2}t^{p(x,y)} \frac{|\varphi(x)-\varphi(y)|^{p(x,y)}}{2p(x,y)|x-y|^{N+sp(x,y)}}\,dx\,dy
    +\int_{\mathcal{C}\Omega}t^{\overline{p}(x)} \frac{\beta(x)|\varphi|^{\overline{p}(x)}v}{\overline{p}(x)}\,dx\\
    &\quad +\int_{\Omega}t^{\overline{p}(x)}\frac{|\varphi|^{\overline{p}(x)}}{\overline{p}(x)}\,dx
    - \lambda \int_{\Omega}t^{q(x)}\frac{V(x)}{q(x)}|\varphi|^{q(x)}\,dx\\
    &\leq t^{p_{0}^{-}}\l(
    \int_{\R^{2N}\setminus(\mathcal{C}\Omega)^2}\frac{|\varphi(x)-\varphi(y)|^{p(x,y)}}{2p(x,y)|x-y|^{N+sp(x,y)}}\,dx\,dy
    +\int_{\mathcal{C}\Omega} \frac{\beta(x)|\varphi|^{\overline{p}(x)}v}{\overline{p}(x)}\,dx +\int_{\Omega}\frac{|\varphi|^{\overline{p}(x)}}{\overline{p}(x)}\,dx\r)\\
    &\quad -\lambda t^{q_{0}^{-}+\eps_{0}}  \int_{\Omega}\frac{V(x)}{q(x)}|\varphi|^{q(x)}\,dx<0.
    \end{align*}
    This shows the assertion.
\end{proof}

Now we are ready to state our main existence result.
\begin{theorem}
    Let hypotheses (S), (P), (F) and ($\beta$) be satisfied. Then there exists $\lambda^*>0$ such that for any $\lambda \in (0,\lambda^*)$ there exists at least one weak solution $u_\lambda \in X$ of problem \eqref{eq}.
\end{theorem}

\begin{proof}
    Let $\lambda^*$ be defined as in Lemma \ref{pos} and choose $\lambda \in (0,\lambda^*)$. Again, invoking Lemma \ref{pos}, we can deduce that
    \begin{align*}
    \inf_{u\in \partial B(0,\rho)}I_{\lambda}(u)>0.
    \end{align*}
    On the other hand, by Lemma \ref{neg}, there exists $\varphi\in X$ such that $I(t\varphi)<0$ for all $t>0$ small enough. Moreover, by Lemma \ref{modular}, for $\|u\|_X<\rho$, we have
    \begin{align*}
    I(u)\geq \frac{1}{p^{+}3^{p^{+}-1}}\|u\|_{X}^{p^{+}}-\lambda\frac{2 \alpha^{q^{-}}\|V\|_{r(\cdot)}}{q^{-}}\|u\|_{X}^{q^{-}},
    \end{align*}
    see the proof of Lemma \ref{pos}. It follows that
    \begin{align*}
    -\infty<m= \inf_{u\in B(0,\rho)}I(u)<0.
    \end{align*}
    Applying Ekeland's variational principle to the functional $I\colon B(0,\rho)\to \R$, we can find a (PS)-sequence $(u_{n})_{n \in \N}\subseteq  B(0,\rho)$, that is,
    \begin{align*}
    I(u_{n})\to m \quad \mbox{and}\quad I'(u_{n})\to 0.
    \end{align*}
    It is clear that $(u_{n})_{n \in \N}$ is bounded in $X$. Thus there exists $u_\lambda\in X$ such that, up to a subsequence, $u_{n} \rightharpoonup u_\lambda$ in $X$. Using Proposition \ref{compact}, we see that $(u_{n})_{n \in \N}$ strongly converges to $u_\lambda$ in $L^{q(\cdot)}(\Omega)$. So, by H\"older's inequality and Proposition \ref{compact}, we obtain that
    \begin{align*}
    \lim_{n\to  +\infty} \int_{\Omega}V(x)|u_{n}|^{q(x)-2}u_{n}(u_{n}-u_\lambda)\,dx=0.
    \end{align*}
    On the other hand, since $(u_{n})_{n \in \N}$ is a (PS)-sequence, we infer that
    \begin{align*}
    \lim_{n\to  +\infty}\l\langle I'(u_{n})- I'(u_\lambda), u_{n}-u_\lambda\r\rangle=0.
    \end{align*}
    Combining this with Lemma \ref{ss}(iii), we can now conclude that $u_{n}\to u_\lambda$ in $X$. Hence,
    \begin{align*}
    I(u_\lambda)=m<0 \quad \text{and} \quad I'(u_\lambda)=0.
    \end{align*}
    We have thus shown that $u_\lambda$ is a nontrivial weak solution for problem \eqref{eq} whenever $\lambda \in (0,\lambda^*)$. This completes the proof.
\end{proof}


\end{document}